\documentclass[12pt,a4paper]{article}
\usepackage[utf8]{inputenc}
\usepackage[T1]{fontenc}
\usepackage{amsmath}
\usepackage{amsfonts}
\usepackage{amssymb}
\newcommand{\ds}{\displaystyle}

\usepackage{graphicx}
\usepackage{epstopdf}
\newtheorem{defn}{Definition}[section]
\newtheorem{defns}{Definitions}[section]
\newtheorem{theo}{Theoreme}[section]
\newtheorem{lemm}{Lemma}[section]
\newtheorem{coro}{Corollary}[section]
\newtheorem{prop}{Proposition}[section]

\newtheorem{remq}{Remarque}[section]
\newenvironment{proof}{\textbf{Proof:}}{\hfill$\square$}

\newcommand{\Ess}{\emph{Ess}}
\newcommand{\Es}{\emph{Es}}
\newcommand{\Ind}{\mathbb I}
\usepackage{graphicx}
\usepackage{epstopdf}
\begin{document}
\thispagestyle{empty}


\title{On nilpotency in Leibniz algebras}

\author{C\^ome J. A. B\'ER\'E\footnote{: bere\_jean0@yahoo.fr}, M. Fran\c coise OUEDRAOGO\footnote{: omfrancoise@yahoo.fr}\\
 and Moussa OUATTARA\footnote{: ouatt\_ken@yahoo.fr}\\%
Laboratoire T.      N.      AGATA 
       /UFR-SEA\\
        Department of Mathematics /  
University of Ouagadougou%
\\
      03 B.      P.      7021 Ouagadougou, Burkina Faso 03\\
}
%
%
\maketitle
\begin{abstract}
The main result  of this paper is to prove that if a (right) Leibniz algebra  $L$ is \textit{right nilpotent} of degree 
$n$, then $L$ is \textit{strongly nilpotent} of degree less or equal to $4n^2-2n+1$.\par 
\centerline{\textbf{R\'esum\'e}} \par 
Nous prouvons    que toute alg\`ebre de Leibniz (droite) $L$ \textit{ nilpotente \`a droite} d'indice $n$ est  \textit{fortement nilpotente} d'un indice inf\'erieur ou \'egal \`a \mbox{$4n^2-2n+1$}.      
\end{abstract}


\section{Introduction}
In \cite{berepilouat} it  is proved that a Malcev algebra is strongly nilpotent if and only if it is right nilpotent. So for Malcev algebras right nilpotency, left nilpotency and strong nilpotency are equivalent to nilpotency. Since Malcev algebra is anti-commutative, right nilpotency and left nilpotency  are equivalent.  This result fails for Leibniz algebras, see for example \cite[Exemple 3.3]{berpilkob}, which is  left nilpotent and not right  nilpotent.\par
Using the notion of $\Es_k$-right nil (or $\Es_k$-left nil) we prove that if an ideal $B$ of a (right) Leibniz algebra  $L$ is 
\textit{right nilpotent} of degree $n$, then $B$ is \textit{strongly nilpotent} of degree less or equal to $4n^2-2n+1$.       

In section \ref{prelim}, we give  some definitions that we will used along the paper, then in the  section \ref{motdroit}, 
 we prove some results on right products of length $n$.
 The section \ref{poids} is devoted to right products of weight $n$ in the ideal $B$ and in the section \ref{mainresult}, we give the main results.

\section{Preliminairies}\label{prelim}
Throughout this paper, $F$ will be a field of characteristic not $2$. All vector spaces and algebras will be finite
dimensional over $F$. Let $n$ be a nonnegative integer, and let us denote  the set  $\left\{ 1,2,\cdots,n\right\}$ by   $\mathbb{I}(n)$.
\addtocounter{defns}{1}
\begin{defn} {(Leibniz algebra)} \cite{Loday,berpilkob}\par 
A Leibniz algebra is a vector space $L$ equipped with a bilinear
map $[-,-]:L\times L\longrightarrow L$, satisfying the Leibniz identity:

\begin{equation}
[x,[y,z]]=[[x,y],z]-[[x,z],y]\textnormal{ for any }x,y,z\in L.\label{droit}
\end{equation}
\end{defn}

If the condition $[x,x]=0$ is fulfilled, the Leibniz identity is
equivalent to the so-called Jacobi identity. Therefore Lie algebras
are particular cases of Leibniz algebras. Algebras which satisfy (\ref{droit}) are also call right Leibniz algebras 
and left Leibniz algebras are defined as followed:

\addtocounter{defns}{1}
\begin{defn} {(left Leibniz algebra)} \cite{berpilkob}\par 
A left Leibniz algebra is a vector space $L$ equipped with a bilinear
map $[-,-]:L\times L\longrightarrow L$, satisfying the left Leibniz identity:

\begin{equation}
[x,[y,z]]=[[x,y],z]+[y,[x,z]]\textnormal{ for any }x,y,z\in L.\label{droit2}
\end{equation}
\end{defn}

It follows from the Leibniz identity (\ref{droit}) that in any Leibniz algebra one
has 
\[
[y,[x,x]]=0,\,[z,[x,y]]+[z,[y,x]]=0,\textnormal{ for all }x,y,z\in L.
\]
\addtocounter{defns}{1}
\begin{defn} 
A subspace $H$ of a Leibniz algebra $L$ is called left (respectively
right) ideal if for $a\in H$ and $x\in L$ one has $[x,a]\in H$
(respectively $[a,x]\in H$). If $H$ is both left and right ideal,
then $H$ is called (two-sided) ideal.
\end{defn}

Let us denote the product $[a,b]$ by $ab$ for all $a,b$ in $L$. $\Ess(L)$ will be the ideal generated by all 
the squares of the  elements of $L$.

For  an ideal  $B$ of a Leibniz algebra  $L$, we introduce the notations and following terminologies:

 Let  $P$ be a product of $m$ factors $s_m, s_{m-1}, \cdots, s_1$, that have been associated in an arbritary way. 
We suppose that $n$ or more factors belong to $B$. We say that the product $P$ is of length $m$ and of weight $n$ with respect 
to the ideal $B$ or more simply that $P$ is of length $m$ and of weight $n$ in $B$. 
The length $m$ of $P$ will be noted $\#\left(P\right)$ and its weight $n$ 
 will be noted $\#_{B}\left(P\right)$.

When $P=\left(\left(\cdots\left(\left(s_ms_{m-1}\right)s_{m-2}\cdots\right)s_3\right)%
s_2\right)s_1$   where the association is made
always right, we say that $P$ is a  \textit{right product} and we write 
$P=s_ms_{m-1}s_{m-2}\cdots s_1$. Similarly, if 
$P=s_1\left(s_2\left(s_3\left(\cdots s_{m-2}\left(s_{m-1}s_m\right)\right)\cdots\right)\right)$    where the association is made
always left, we say that $P$ is a  \textit{left product}.\\
Let $S_1,S_2,\cdots,S_p$ be right products. One can  write the right product (with $S_j$ as factors)  $N=S_pS_{p-1}\cdots S_1$.
We call $N$ a standard product.\\ 
\addtocounter{defn}{1}
\begin{defns}
\begin{itemize}
\item  A subspace $B$ of the underlying vector space  $L$ is \textit{right nilpotent}  if $B^n=\{0\}$ for some $n\geq1$,
 where $B^1=B$ and $B^{n+1}=B^n\cdot B$. By convention, we set $B^0 = L$. Notice that $B^n$ is generated by 
\textit{right products} of length $n$ and weight $n$ in $B$.  
\item  A subspace $B$ of the underlying vector space  $L$ is %
\textit{left nilpotent}  if\, $^n\!B=0$ for some $n\geq1$,  where $^1B=B$ and $^{1+n}\!B=B\cdot(^n\!B)$. By convention, 
we set $^0B= L$. Notice that $^n\!B$ is generated by \textit{left products} of length $n$ and weight $n$ in $B$. 
\end{itemize}
\end{defns} 
\addtocounter{defn}{1}
\begin{defns}
\begin{itemize}
\item  Let $B^{\left\{n\right\}}$ be the subspace of the underlying vector space  $L$ generated by all the products of length $n$ in $B$, associated in arbritary way.  We say that the ideal $B$ is \textit{nilpotent}  if there exists an integer $n$ such that $B^{\left\{n\right\}}=\left\{0\right\}$.      
\item Let $B^{\left\langle n\right\rangle }$ be the subspace generated by all products of elements in  $L$ with at least $n$ elements in $B$. 
A subspace $B$  is \textit{strongly nilpotent}  if $B^{\left\langle n\right\rangle }=\left\{ 0\right\}$ for some $n\geq1$.
\end{itemize}
\end{defns} 

Naturally,
$B^{\left\langle n\right\rangle }$ is an ideal of  $L$ and one has\\   $B\supseteq B^{\left\langle 1\right\rangle }\supseteq B^{\left\langle 2\right\rangle }\supseteq\cdots\supseteq B^{\left\langle n\right\rangle }\supseteq\cdots$
and $B^{\left\langle i\right\rangle }B^{\left\langle j\right\rangle }\subseteq B^{\left\langle i+j\right\rangle }$
for all nonnegative integers $i,j\geq1$.\par


\addtocounter{defn}{1}
\begin{defns}
\begin{itemize} 
\item Let $D$ be a subspace of the underlying vector space  $L$. Let $k$ be a nonnegative integer. $D_{(L,k)}$ is the vector subspace
generated by all  right products: "$da_{k}\cdots a_3a_2a_1$" where $d$ belongs to $D$ and $a_i$ belongs to  $L$ for any integer $i\in\Ind(k)$. 
\item  Let $D$ be a subspace of the underlying vector space  $L$. Let $k$ be a a nonnegative integer. $_{(L,k)}D$ is the vector subspace
generated by all  left products: "$a_1(a_2(a_3(\cdots(a_{k}d)))\cdots) $" where $d$ belongs to $D$ and $a_i$ belongs to  $L$ for any integer  $i\in\Ind(k)$.
\end{itemize}
\end{defns}     

\addtocounter{defn}{1}
\begin{defns} 
\begin{itemize}
 \item Let $B\neq\{0\}$ be an ideal of the Leibniz algebra $L$.     
If there is an integer $k\geq1$, such that the ideal $\Es(B)=B\cap\Ess(L)$ satisfies  
$\Es(B)_{(L,k)}=\Es(B)\underset{k\textnormal{ times}}{\underbrace{AAA\cdots A}}=\left\{ 0\right\} $, 
we will say that $B$ is \textit{$\Es_k$-right nil}. 
 \item Let $B\neq\{0\}$ be an ideal of the Leibniz algebra $L$.     
If there is an integer $k\geq1$, such that the ideal $\Es(B)=B\cap\Ess(L)$ satisfies  
$_{(L,k)}\Es(B)=\underset{k\textnormal{ times}}{\underbrace{A\cdots AAA}}\Es(B)=\left\{ 0\right\} $, 
we will say that $B$ is \textit{$\Es_k$-left nil}.   
\end{itemize}
\end{defns} 

\addtocounter{defns}{1}
\begin{defn} Let  $L$ be a Leibniz algebra and $B$ an ideal 
in  $L$. Let $a,b$ ly in  $L$.  If $a-b\in B$, we will say that $a\equiv b \textnormal{ (modulo }B)$.%
\end{defn}

\section{Right products in the  Leibniz algebra $L$}\label{motdroit}

\begin{lemm}\label{drf0} Let $L$ be a Leibniz algebra and $B$ an ideal of  $L$. For all $a\in L$ and  for all $b\in B$,  
$ab+ba\in \Es(B)=B\cap\Ess(L)$.
\end{lemm}
\begin{proof} Obvious.   
\end{proof}  
  
\begin{lemm}\label{drf1}  For any integer $n\geq 1$, $B^n\subseteq\, ^n\!B+\Es(B)$.
\end{lemm}

\begin{proof} For $n=1$ or $2$, the result is obvious.
Let $n=3$ and 
 $a,b,c$ be elements of $B$, we have $a(bc)+(bc)a\in\Es(B)$ and so $(bc)a$ equals $a(bc)$ modulo $\Es(B)$. So, we can write  $B^3\subseteq\, ^3\!B+\Es(B)$. Let us set by hypothesis that $B^p\subseteq\, ^p\!B+\Es(B)$ for all integer $p\leq n$ and prove that $B^{n+1}\subseteq\, ^{1+n}\!B+\Es(B)$.  We have\\
\begin{align*}
B^{n+1} & =\left(B^{n-1}\cdot B\right)\cdot B
  \subseteq B\cdot\left(B^{n-1}\cdot B\right)+\Es(B)\\
  &\subseteq B\cdot\left[B\cdot B^{n-1}+\Es(B) \right]+\Es(B)
  \subseteq B\cdot(B\cdot B^{n-1})+\Es(B)\\
  &\subseteq B\cdot(B\cdot \left[^{-1+n}B+\Es(B) \right])+\Es(B)\\
  &\subseteq B\cdot(B\cdot^{-1+n}\!\!\!B)+\Es(B)
  \subseteq\, ^{1+n}\!B+\Es(B).
  \end{align*}
 \end{proof}

\begin{lemm}\label{drf}  Let $L$ be a Leibniz algebra and $B$ an ideal of  $L$. Let us define $B_{0}=L$, $B_1=B$ and
$B_{k}=B^{k}+\Es(B)$ for all integer $k\geq2$ ; $B_{k}$
is an ideal of  $L$, which satisfies  $B_{k}\supseteq B_{k+1}$.
\end{lemm}       

\begin{proof} It is  known that $\Es(B),B_{0},B_1$ are ideals. Let us 
assume 
 that for an integer $k\geq2$, $B_{k}$ is an ideal. Then one has  $B_{k}\cdot A\subseteq B^{k}+\Es(B)$ and $A\cdot B_{k} \subseteq B^{k}+\Es(B)$.  Let us show that $B_{k+1}$ is also an ideal.      Indeed; 
\begin{align*}
B_{k+1}\cdot A & =\left(B^{k+1}+\Es(B)\right)\cdot  A
 \subseteq B^{k+1}\cdot A+\Es(B)\\
 & \subseteq\left(B^{k}\cdot     B\right)\cdot A+\Es(B)
 \subseteq A\cdot\left(B^{k}\cdot     B\right)+\Es(B) \textnormal{ (see     Lemma \ref{drf0})}\\
 & \subseteq \left(A\cdot B^{k}\right)\cdot B+\left(A\cdot B\right)\cdot B^{k}+\Es(B)\\
 &\subseteq \left(B^k+\Es(B)\right)\cdot B+ B\cdot B^k+\Es(B)\\
 &\subseteq B^{k+1}+B\cdot B^k+\Es(B)
 \subseteq B^{k+1}+B\cdot\left(^{k}\!\!B+\Es(B)\right)+\Es(B)\\ 
 &\subseteq B\cdot ^{k}\!\!B+B^{k+1}+\Es(B)
 \subseteq ^{1+k}\!\!\!B+B^{k+1}+\Es(B)\\ 
 &\subseteq B^{k+1}+\Es(B) =B_{k+1} \textnormal{ (see Lemma \ref{drf1})}.    
\end{align*}
Thanks to Lemma \ref{drf0}, we have $A\cdot B_{k+1}\subseteq B_{k+1}\cdot A+\Es(B)$. So we obtain  $A\cdot B_{k+1}\subseteq B^{k+1}+\Es(B) =B_{k+1}$.
\end{proof}

\begin{prop} \label{prop:P-som} 
For a given right product $P_{0}=a_{m}a_{m-1}a_{m-2}\cdots a_3a_2a_1$  and 
$Q_{0}$ an arbritary product of length  $m'$,
let us set recursively, for any integer 
$i\in\Ind(m-1)$,
 $P_{i}=a_ma_{m-1}\cdots a_{i+1}=\ds\prod_{j=1}^{m-i}a_{m-j+1}$ and 
  $Q_i=-Q_{i-1}a_i$.
Then: 
$$
T_{m}=Q_{0}P_{0}=\sum_{i=1}^{m-1}Q_{i-1}P_{i}a_{i}+Q_{m-1}a_{m}.$$
\end{prop}

\begin{proof} First of all let us define the following products:
For $i\in\Ind(m-1)$, we set  $Q'_{i-1}=Q_{i}$ and $a'_{m-i+1}=a_{m-i+2}$. It follows that  
$$P'_i=\ds\prod_{j=1}^{m-i}a'_{m-j+1}=\ds\prod_{j=1}^{m-i}a_{m-j+2}=\prod_{j=1}^{m+1-(i+1)}a_{m+1-j+1}=P_{i+1}.$$
Now 
if   $m=2$, then\\
$\begin{aligned} T_2&=Q_{0}\left(a_2a_1\right) =Q_{0}a_2a_1-Q_0a_2a_1=Q_{0}P_1a_1+Q_1a_1,\textnormal{ and if  }    m=3;\\
T_3&=Q_{0}\left(a_3a_2a_1\right)\\ 
& =  Q_{0}\left(a_3a_2\right)a_1-\left(Q_{0}a_1\right)\left(a_3a_2\right)=Q_0P_1a_1+Q_1\left(a_3a_2\right)\\
 & =  Q_0P_1a_1+Q_1a_3a_2-Q_1a_2a_3\\
 & = Q_0P_1a_1+Q_1P_2a_2+Q_2a_3.
\end{aligned}
$\par
Assume 
that  
 \begin{equation}\label{SionPm}
     T_{m}=\sum_{i=1}^{m-1}Q_{i-1}P_{i}a_{i}+Q_{m-1}a_{m}.     
    \end{equation} 

Then for 
$P=a_{m+1}a_{m}a_{m-1}\cdots a_3a_2a_1$, we have:

$\begin{aligned}T_{m+1} & =  Q_{0}\ds\prod_{k=1}^{m+1}a_{m-k+2}
 =Q_{0}\left[\ds\prod_{k=1}^{m}a_{m-k+2}a_1\right]\\ 
 & = Q_{0}\ds\prod_{k=1}^{m}a_{m-k+2}a_1-Q_0a_1\ds\prod_{k=1}^{m}a_{m-k+2}\\
& = Q_{0}\ds\prod_{k=1}^{m}a_{m-k+2}a_1+Q_1\ds\prod_{k=1}^{m}a_{m-k+2} \\
\end{aligned}$\par 

It follows that 

$\begin{aligned}
T_{m+1}  &  =Q_{0}\left(P'_0a_1\right)
= Q_{0}P'_0a_1+Q_1P'_0\\
 & = Q_{0}P_1a_1+Q'_0P'_0 \label{equ3} 
\end{aligned}$\par 

Since the  length of $P'_0$ is $m$ we can  write:

$\begin{aligned}
Q'_0P'_0 & =\sum_{i=1}^{m-1}Q'_{i-1}P'_ia'_{i}+Q'_{m-1}a'_{m}\\
 & =\sum_{i=1}^{m-1}Q_{i}P_{i+1}a_{i+1}+Q_{m}a_{m+1}.
\end{aligned}$\par

Thanks to Equation (\ref{equ3}), we obtain:

$\begin{aligned}
T_{m+1} & = Q_{0}P_1a_1+Q'_0P'_0\\
	& = Q_{0}P_1a_1+\sum_{i=1}^{m-1}Q_{i}P_{i+1}a_{i+1}+Q_{m}a_{m+1}\\
 & = \sum_{i=1}^{m}Q_{i-1}P_{i}a_{i}+Q_{m}a_{m+1}.
\end{aligned}$\\
\hfill\end{proof}

\begin{remq}\label{remq} With the hypothesis of Proposition \ref{prop:P-som}, 
%
note $p,p',p''$ the respective weight of $P,Q_{i-1},P_{i}\;(1\leq i\leq m )$
with regard to the ideal $B$.      Let us consider the following table:

\begin{table*}[h]
\begin{tabular}{||l||l||l||l||l}
\hline 
 $P_{m-1}=a_{m}$  &   & $\cdots$  & $P_{i-1}=P_{i}a_{i}$  & $\cdots$ \tabularnewline
\hline 
 $Q_{m-1}=-Q_{m-2}a_{m-1}$  &   & $\cdots$  & $Q_{i-1}=-Q_{i-2}a_{i-1}$  & $\cdots$ \tabularnewline
\hline 
\end{tabular}
\end{table*}

\begin{table*}[h]%
\begin{tabular}{l||l||l||l||l||l||}
\hline 
$\cdots$  & $P_{i+j}=P_{i+j+1}a_{i+j+1}$  & $\cdots$  &   & $P_1=P_2a_2$  & $P_{0}=P_1a_1$ \tabularnewline
\hline 
$\cdots$  & $Q_{i+j}=-Q_{i+j-1}a_{i+j}$  & $\cdots$  &   & $Q_1=-Q_0a_1$  & $Q_{0}=Q_{0}$ \tabularnewline
\hline 
\end{tabular}\label{alg5} 
\end{table*}

Let 
$\Lambda=\left\{(a_i)_{1\leq i\leq m},(b_j)_{1\leq j\leq m'}\right\}$ be the set of all factors of  $P$. It is easy to check that $\Lambda$ also produces $\left(Q_{k-1}P_k\right)a_k,Q_{m-1}a_{m}$ for $k\in{\mathbb I}(m-1)$. Then, one has for an integer 
 $k$ in $\Ind(m-1)$:
\begin{eqnarray}
\#\left(P\right)\phantom{_B}&=&\#\left(Q_{k-1}\right)+\#\left(P_k\right)+1%
\label{rem1.1}\\
\#_{B}\left(P\right)&=&\#_{B}\left(Q_{k-1}\right)+\#_{B}\left(P_k\right)+
 \#_{B}\left(a_k\right)\label{rem1.2} \\
 \#\left(P\right)\phantom{_B}&=&\#\left(Q_{m-1}\right)+1\label{rem1.3}\\
\#_{B}\left(P\right)&=&\#_{B}\left(Q_{m-1}\right)+
 \#_{B}\left(a_m\right).\label{rem1.4}
\end{eqnarray}         
\end{remq}

\begin{lemm} \label{lem:nilpo1} Any product $T$ with length $m$
in a Leibniz algebra $L$ is a linear combination of right products of length $m$.
\end{lemm}    
 
\begin{proof} By induction on the length $m$, we have: 
If $m$ equals $1$ or $2$, there is nothing to do.
If $m=3$, one can notice that $T=abc$ or $a(bc)=abc-acb$ for all $a,b,c$ in $L$. The lemma is also obvious.\\
Let us suppose that the lemma is true for a product which length is strictly less than $m\geq4$.

Now for a given product (with length $m$) $T=Q_0P_0$ where  $P_{0}$ is a right  product of  length $n$ such that $m>n\geq1$.  
Thanks to Proposition \ref{prop:P-som}, 
\[
T=Q_{0}P_{0}=\sum_{i=1}^{n-1}Q_{i-1}P_{i}a_{i}+Q_{n-1}a_{n}.     
\]
The length of following products $Q_{i-1}P_{i}$ for $i\in\mathbb{I}(n-1)$  and $Q_{n-1}$ is $m-1$, so they are linear combinations of right products of length $m-1$.     
Then  $T$ is a linear combination of right products of length $m$.
\end{proof}

\section{Right products of weight $n$  in the ideal $B$}\label{poids}

\begin{lemm} \label{lienon-1}  Any product $T$ with length $m$
and weight $n$ with regard to the ideal $B$ of Leibniz algebra $L$ is a linear combination of right products of length $m$ and weight $n$.
\end{lemm}    

\begin{proof} 
By induction on the length $m$, we have:\\ 
The lemma is obvious if $m\leq2$.      
If $m=3$, then there are $a_1,a_2,a_3$ elements
in  $L$ such that $P$ is one of the following linear combinations of right \mbox{products:} $a_3a_2a_1$ 
and $a_3\left(a_2a_1\right)
=a_3a_2a_1-a_3a_1a_2$. So for $m=3$, $P$ is a linear combination of right products.\par

Let us suppose that the lemma is true for a product which length is strictly less than $m\geq4$.

Now for a given product (with length $m$ and weight $n\leq m$) $T=Q_0P_0$ where  $P_{0}$ is a right  product of  length $m'$ such that $m>m'\geq1$.  
Thanks to Proposition \ref{prop:P-som}, 
\[
T=Q_{0}P_{0}=\sum_{i=1}^{m'-1}Q_{i-1}P_{i}a_{i}+Q_{m'-1}a_{m'}.     
\]
The length of following products $Q_{i-1}P_{i}$ for $i\in\mathbb{I}(m'-1)$ and $Q_{m'-1}$ is $m-1$, so they are linear combinations of right products of length $m-1$.     
Then  $T$ is a linear combination of right products of length $m$.

Thanks to the equations; Equation (\ref{rem1.1}) to  Equation (\ref{rem1.4}), it is clear that for $i\in\Ind(p-1)$, the weight of the standart product $Q_{i-1}P_{i}a_{i}$ and the weight of $Q_{p-1}a_{p}$ are equal to $n$. 
\end{proof}

\begin{lemm} \label{Bn} Let  $L$ be a Leibniz algebra and $B\neq\{0\}$ an ideal in  $L$.      Let $P_{0}=a_{m}a_{m-1}a_{m-2}\cdots a_3a_2a_1$
be a right product with length $m$ and weight $n\geq1$ in $B$. Then $P_{0}$ belongs to $B_n$.     
\end{lemm}

\begin{proof} 
Let $\sigma$ be an  injective map of  $\mathbb{I}(n)$ to
dans $\mathbb{I}(m)$ such that $i<j$ implies that $\sigma(i)<\sigma(j)$ and for all  $j\in\mathbb{I}(n),a_{\sigma(j)}\in B$.      
Let us also define for all integer $k\in\mathbb{I}(n-1)$,
 the following products: 
$$\begin{aligned}
Q'_{0}=&a_{m}a_{m-1}\cdots a_{\sigma(n)+1},\\ 
Q_k=&Q'_{k-1}a_{\sigma(n-k+1)}, \\                   
Q'_k=&Q_ka_{\sigma(n-k+1)-1}\cdots a_{\sigma(n-k)+1},\\ 
Q_n=&Q'_{n-1}a_{\sigma(1)}\cdots a_1.                       
\end{aligned}$$
 
Clearly, $Q_1$ belongs to $B\subseteq B^{1}+\Es(B)=B_1$  and also   $Q'_1$ belongs to $B$.\\
Since $B$ is an ideal $Q_2=Q'_1a_{\sigma(n-j+1)}$ belongs to $B\cdot B\subseteq B^{2}+\Es(B)=B_2$.     \\
By induction, let us suppose that for integer $j\in\Ind(n-1)$, we have $Q_{j}$ belongs to $B_{j}= B^{j}+\Es(B)$. Then let us show that 
 $Q_{j+1}$ is an element of $B_{j+1}=B^{j+1}+\Es(B)$. Indeed, we have,\\
 $Q'_{j}=Q_{j}a_{\sigma(n-j+2)-1}\cdots a_{\sigma(n-j+1)+1}$ is an element of the ideal $B_{j}$ and so on,\\
 $Q_{j+1}=Q'_{j}a_{\sigma(n-j+1)}$ belongs to 
$B_{j}\cdot B\subseteq B^{j+1}+\Es(B)=B_{j+1}$.\\
Then we have proved that $P_{0}=Q_{n}$ is an element of $B_{n}$.     
\end{proof}

\begin{lemm} \label{lem:laqest} Let $k,\ell$ be integers such that $1\leq k\leq\ell$  and let  $L$ be a Leibniz algebra 
and $B$ an ideal of $L$ which is $\Es_{k}$-right nil.  A right product  $P=a_{m}a_{m-1}a_{m-2}\cdots a_3a_2a_1$,
of length $m$ and weight (in $B$) $n$ greater or equal to $2\ell$, belongs to $B_{(L,k)}^{\ell}$.      
\end{lemm}

\begin{proof}  The right product $Q=a_{m}a_{m-1}a_{m-2}\cdots a_{k+1}$ is of weight greater or equal to $\ell$, indeed let $n'$ be the weight of  $Q$ and $n''$ be the weight of $a_{k}\cdots a_3a_2a_1$.   We have 
$0\leq n''\leq k$ and the equality $P=Qa_{k}\cdots a_3a_2a_1$ implies that
 $n'\leq n\leq k+n'$. So $n'\geq n-k\geq2\ell-k\geq\ell$.     
The Lemma \ref{Bn} tells that $Q\in B_{\ell}$.      And so on, $P=Qa_{k}\cdots a_3a_2a_1\in\left(B_{\ell}\right)_{(L,k)}=\left(B^{\ell}+\Es(B)\right)_{(L,k)}=B^l_{(L,k)}$
since $B$ est un ideal $\Es_k$-right nil.
\end{proof}

\begin{lemm} \label{lem:2n+1} Let $k$ be an integer such that the ideal  $B$ is $\Es_{k}$-right nil.  Let $P$ be product of weight  $t\geq4k^2-2k+1$ with regard to the ideal $B$. then $P$ is a linear combination of right products  $Q_j$ ($P=\ds\sum_{j \textnormal{ fini}}\mu_{j}Q_{j}$)  
such that, for any $j$, we have $Q_{j}$ belongs to $\left(B^{k}\right)_{(L,k)}$ or has at least one factor in $\left(B^{k}\right)_{(L,k)}$.     
\end{lemm}

\begin{proof}  Let $k>1$ and $t\geq4k^2-2k+1$.      Thanks to Lemma \ref{lienon-1}, any product $P$ of weight greater or equal to $t$ is a linear combination of right products of weight greater or equal to $t$.      Let $P=\ds\sum_{j}\mu_{j}Q_{j}$ where $Q_{j}$ is a right product of weight greater or equal to $t$.\par      
For any $j$ we have $Q_{j}=s_{j,p}s_{j,p-1}\cdots s_{j,1}$ where $s_{j,i}\in L$ ($p$ is the length of $Q_{j}$).      
\begin{description}
\item[-] if, there is one element $s_{j,i_{0}}$ such that it's weight is greater or equal to $2k$, then $s_{j,i_{0}}$ belongs to $\left(B^{k}\right)_{(L,k)}$ by application of Lemma \ref{lem:laqest}.     And so on  $Q_{j}$ has a factor in  $\left(B^{k}\right)_{(L,k)}$. 
\item[-] else, every factor $s_{j,i}$ has a weight strictly less than $2k$.      Let $q$ be the number of factors $s_{j,i}$ with a weight greater or equals to
 $1$.     
Then one has $q\left(2k-1\right)\geq t=4k^2-2k+1$ and then $q>2k$.\\
When the weight of  $s_{j,i}$ is greater or equal to $1$, we have $s_{j;i}$ belongs to  $B$.     
So $Q_{j}$ is of weight greater or equal to  $q$. 
Since $q\geq2k$, we have $Q_{j}\in\left(B^{k}\right)_{(L,k)}$
thanks to Lemma \ref{lem:laqest}. 
\end{description} 
     For language simplification, we will say that $Q_j$ has at least one factor in $\left(B^{k}\right)_{(L,k)}$.
\end{proof}

\section{Main Theorem}\label{mainresult}

\begin{theo}\label{theo} Let $k'$ be a nonnegative integer,  let $L$ be a Leibniz algebra and let 
 $B$ be an ideal of  $L$ which is $\Es_{k'}$-right nil.  Then the following assertions are equivalents: 
\begin{description}
\item [{(i)}] $B$ is \textit{right nilpotent} ; 
\item [{(ii)}] $B$ is \textit{nilpotent} ;
 
\item [{(iii)}] $B$ is \textit{strongly nilpotent}.      ;
\end{description}
\end{theo}

\begin{proof} Indeed, for any integer $k\geq1$, the vectors spaces's inclusions   $B^{k}\subseteq B^{\{k\}}\subseteq B^{\left\langle k\right\rangle }$ tell us that $(iii)\Rightarrow(ii)\Rightarrow(i)$.\\
 Furthermore, suppose that there  is an integer  $\ell'\geq1$ such that 
$B^{\ell'}=\left\{ 0\right\} $. Let us define $k=\max\left\{ k',\ell'\right\} $, then for an integer $ell$ such that  $\ell\geq4k^2-2k+1$, the Lemma \ref{lem:2n+1} tells us that any product $P$ with weight  greater or equal to $\ell\geq4k^2-2k+1$,  in $B$ is a linear combination of right products which have at least one factor in $\left(B^{k}\right)_{(L,k)}\subseteq\left(B^{k'}\right)_{(L,k)}=\left\{ 0\right\} $.     
And so on $P=0$. Then $B^{\left\langle \ell\right\rangle }=0$.     
The implication $(i)\Rightarrow(iii)$ is done. 
\end{proof}

\begin{coro}\label{equiv} Let $L$ 
be a  Leibniz algebra. 
The following assertions are equivalents:  
\begin{description}
\item [{(i)}]  $L$ is \textit{right nilpotent} ; 
\item [{(ii)}]  $L$ is \textit{nilpotent} ;
\item [{(iii)}]  $L$ is \textit{strongly nilpotent}.      
\end{description} 
\end{coro}
\begin{proof} Clearly we have $(iii)\Rightarrow(ii)\Rightarrow(i)$.\\
Notice that the ideal $\Ess(A)$ is a subset of  $A^2$.
Assume that we have  $(i)$ and then let us show that  $(iii)$ is verified.\\
We know that there is an integer $\ell'>1$ which satisfies that $A^{\ell'}=\{0\}$. So we have  \\
$\Ess(A)\underset{\ell'-2\textnormal{ times}}{\underbrace{AAA\dots A}}\subseteq
(A^2)\underset{\ell'-2\textnormal{ times}}{\underbrace{AAA\dots A}}\subseteq A^{\ell'}=\{0\}$. So $A$ is $\Es_{\ell'}$-right nil. With part of the proof of the theorem \ref{theo}, we can conclude that $(i)\Rightarrow (iii)$.
\end{proof} 

\begin{remq}\end{remq}
\begin{itemize}
 \item For (right) Leibniz algebra $L$, the ideal $\Ess(L)$ is always $\Es_1$-left nil. But if $\Ess(L)$ is not $\Es_k$-right nil for some integer $k$,  $L$ is not nilpotent.
 \item For (left) Leibniz algebra $L$, the ideal $\Ess(L)$ is always $\Es_1$-right nil. But if $\Ess(L)$ is not $\Es_k$-left nil for some integer $k$,  $L$ is not nilpotent (cf. \cite{bereaslaokonk}).
 \item By duality, we have also proved  that if a (left) Leibniz algebra  $L$ is \textit{left nilpotent} of degree $n$, then $L$ is \textit{strongly nilpotent} of degree less or equal to $4n^2-2n+1$.
\end{itemize}

\end{document}